\documentclass[12pt, reqno]{amsart}
\usepackage{amsmath, amsthm, amscd, amsfonts, amssymb, graphicx}
\usepackage{graphicx} % Required for inserting images
\usepackage{amsfonts}
\usepackage{amssymb}
\usepackage{amsmath}
\usepackage{upgreek}
\usepackage{color}
\usepackage{tikz-cd}

\usepackage{amsthm}
\newtheorem{theorem}{Theorem}
\newtheorem{corollary}{Corollary}
\newtheorem{lemma}{Lemma}
\newtheorem{proposition}{Proposition}
\newtheorem{remark}{Remark}
\usepackage{hyperref}
\usepackage{xcolor}
\usepackage{color,soul}
\usepackage{mathtools}

\begin{document}
\centerline{}

\centerline{}

\title{Orbifold Euler Characteristics of Compactified Jacobians}

\author{Sofia Wood}

\begin{abstract}
We calculate the orbifold Euler characteristics of all the degree \(d\) fine universal compactified Jacobians \(\overline{\mathcal{J}}^d_{g,n}\) over the moduli space of stable curves of genus \(g\) with \(n\) marked points. These are defined in the paper \cite{PT23} by Pagani and Tommasi. We show that this orbifold Euler characteristic agrees with the Euler characteristic of \(\overline{\mathcal{M}}_{0, 2g+n}\) up to a combinatorial factor, and in particular, is independent of the degree d and the choice of degree \(d\) fine compactified universal Jacobian. As a special case, we see that the fine compactified universal Jacobians \(\overline{\mathcal{J}}^d_{g,n}(\phi)\) defined in \cite{KP19}, depending on a universal polarization \(\phi\), have orbifold Euler characteristic independent of \(\phi\).\\

\end{abstract}

\maketitle

\section{Introduction}
\subsection{Background and motivation}
\noindent
Orbifold Euler characteristics have been studied for a long time by mathematicians as a generalization of topological Euler characteristics. Perhaps most famously, in \cite{HaZa}, Harer and Zagier calculate the orbifold Euler characteristics of the moduli spaces of smooth genus \(g\), \(n\)-pointed curves for all \(g\) and \(n\) such that \(2g-2+n>0\), obtaining the formula:
\begin{equation*}
\chi_{orb}(\mathcal{M}_{g,n}) = \frac{B_{2g}(-1)^n(2g-1)(2g+n-3)!}{(2g)!}.
\end{equation*}
They then proceed to use this calculation to compute the ordinary Euler characteristics of \(\mathcal{M}_{g}\) for \(g \geq 2\), which consequently, can be used to deduce facts about the cohomology of \(\mathcal{M}_{g}\). In the paper \cite{BiHa}, Bini and Harer use the Harer-Zagier calculation to compute \(\chi_{orb}(\overline{\mathcal{M}}_{g,n})\) for all \(g,n\) such that \(2g-2+n>0\). More recenty, some examples of orbifold Euler characteristics of the moduli spaces of abelian differentials \(\mathbb{P}\Omega\mathcal{M}_{g,n}(\mu)\) are calculated in the paper \cite{CMZ} by Zachhuber, Costantini and M\"oller. In this paper, we extend these calculations to fine compactified universal Jacobians.\\

\noindent
Over the moduli space of smooth curves \(\mathcal{M}_{g,n}\), there is a smooth, proper morphism \(\mathcal{J}^d_{g,n} \rightarrow \mathcal{M}_{g,n}\) of Deligne--Mumford (DM) stacks where the geometric points of \(\mathcal{J}^d_{g,n}\) are pairs consisting of a genus \(g\) algebraic curve with \(n\) marked points and a line bundle of degree \(d\) on it. The morphism is the forgetful morphism, obtained by forgetting the data of the line bundle. The fibers over geometric points are isomorphic to abelian varieties and therefore have Euler characteristic zero. By properties of the morphism and of orbifold Euler characteristics,
\begin{align*}
    \chi_{orb}(\mathcal{J}^d_{g,n}) = \chi_{orb}(\mathcal{M}_{g,n})\chi_{orb}(F) = 0,
\end{align*}
for \(g>0\) where \(F\) is any fiber of the morphism.\\

\noindent
To obtain a space with non-trivial Euler characteristic, it is natural to consider an extension of this morphism to a proper morphism over the compactification \(\overline{\mathcal{M}}_{g,n}\) of \(\mathcal{M}_{g,n}\).
Finding such an extension is a classical problem and many such extensions have been found, such as those in \cite{Cap94}, \cite{Sim94}, \cite{Pan96}, \cite{Est01} and \cite{Mel16}. In the paper \cite{KP19}, Kass and Pagani construct fine compactified universal Jacobians \(\overline{\mathcal{J}}^d_{g,n}(\phi)\), which depend on a stability condition \(\phi\). In the more recent paper \cite{PT23}, it is shown that these universal compactified Jacobians belong to the more general class of fine compactified universal Jacobians.\\

\noindent
We compute the orbifold Euler characteristics of all fine compactified universal Jacobians over \(\overline{\mathcal{M}}_{g,n}\) such that \(2g-2+n>0\).\\

\noindent
Recently, universal compactified Jacobians have been used in the exploration of tautological classes on the moduli space of stable curves, e.g. in the study of (logarithmic) double ramification cycles (see \cite{MR3851130, LogDR}). This raises important questions, e.g. how to define the tautological ring of these spaces, and lift known tautological relations to this ring (extending work in \cite{MR3477953}). Work in this direction can be seen, for example, in \cite{MR23} and \cite{BaLh}. The work in this paper was partly motivated by the goal of obtaining a better understanding of the global geometry of these compactified Jacobians.\\

\begin{theorem}\label{main theorem}
For all \(d \in \mathbb{Z}\), non-negative integers \(g\) and \(n\) such that \(2g-2+n>0\) and universal fine  compactified Jacobians \(\overline{\mathcal{J}}^d_{g,n}\) of degree \(d\), we have:
\begin{align*}
\chi_{orb}(\overline{\mathcal{J}}^d_{g,n}) = \frac{1}{2^g(g!)}\chi(\overline{\mathcal{M}}_{0,2g+n})
\end{align*}
where \(\chi(\overline{\mathcal{M}}_{0,2g+n})\) is the ordinary topological Euler characteristic of the variety \(\overline{\mathcal{M}}_{0,2g+n}\).
\end{theorem}
\noindent
This may be surprising to some readers, in light of the fact that in Section 6 of \cite{KP19}, it is shown that there exist \(g,n\) and \(d\), along with universal polarizations \(\phi_1\) and \(\phi_2\), such that the corresponding fine compactified universal Jacobians \(\overline{\mathcal{J}}^d_{g,n}(\phi_1)\)  and \(\overline{\mathcal{J}}^d_{g,n}(\phi_2)\) are not isomorphic as Deligne--Mumford stacks.
\subsection{Idea of the proof}
We consider the morphisms \(\overline{\mathcal{J}}^d_{g,n} \rightarrow \overline{\mathcal{M}}_{g,n}\) for a fine compactified universal degree \(d\) Jacobian over the moduli space of stable curves. Such a morphism is a proper, representable morphism of DM stacks and the fiber over a point \([(C;p_1,...,p_n)]\) is a degree \(d\) fine, smoothable compactified Jacobian over \(C\). Such compactified Jacobians over a nodal curve \(C\) are defined in terms of stability conditions \(\sigma\) and we denote the Jacobian by \(\overline{J}^d_{\sigma}(C)\).\\

\noindent
In particular, under the stratification \(\overline{\mathcal{M}}_{g,n} = \underset{\Gamma \in G(g,n)}{\bigsqcup}\mathcal{M}^{\Gamma}\) by dual graph, we show that the topological Euler characteristic of the fiber over a geometric point \([(C;p_1,...,p_n)]\) in \(\mathcal{M}^{\Gamma}\) depends only on \(\Gamma\).
\subsection{Stratifying the fibers and their Euler characteristics}We adapt results from the paper \cite{MeVi}, where a stratification of a compactified Jacobian \(\overline{J}^d_{\sigma_{\phi}}(C)\) corresponding to a stability condition \(\sigma_{\phi}\) which is defined in terms of a polarization \(\phi\) on \(C\) is obtained. The strata in this stratification are all isomorphic to the generalized Jacobian of some curve which is a partial normalization of \(C\) at a subset of its nodes. We show that the same is true for any smoothable, fine compactified Jacobian \(\overline{J}^d(C)\) of a nodal curve \(C\) which is necessarily defined in terms of some stability condition \(\sigma\) on its dual graph \(\Gamma(C)\). We write \(\overline{J}^d(C) = \overline{J}^d_{\sigma}(C)\).\\

\noindent
If \(G(g,n)\) is a set of representatives of automorphism classes of stable graphs of genus \(g\) with \(n\) marked points and \(G(g,n)^0\) is the subset of \(G(g,n)\) consisting of stable graphs whose vertices all have genus 0, we use our stratification of a smoothable fine compactified Jacobian of \(C\) corresponding to an arbitrary stability condition \(\sigma\) on \(\Gamma(C)\) to obtain:
\begin{align*}
\chi_{top}(\overline{J}_{\sigma}^d(C)^{an}) =\chi_{orb}(\overline{J}_{\sigma}^d(C)) = \begin{cases}c(\Gamma(C))\;\;\;\text{if } \Gamma(C) \in G(g,n)^0\\
0\;\;\;\;\;\;\;\; \text{if }\Gamma(C) \in G(g,n)\setminus G(g,n)^0\end{cases}
\end{align*}
where \(c(\Gamma(C))\) is the number of spanning trees of \(\Gamma(C)\).

\subsection{Stratifying and obtaining the orbifold Euler characteristic of a fine compactified universal Jacobian}
We then apply Proposition \ref{constant fiber multiplicativity} to the morphism \(\overline{\mathcal{J}}^d_{g,n}|_{\mathcal{M}^{\Gamma}} \rightarrow \mathcal{M}^{\Gamma}\). Informally, this proposition states that for a morphism of DM stacks \(\mathcal{M} \rightarrow \mathcal{N}\) which is a fibration on some stratification of the base such that all fibers have the same topological Euler characteristic, \(\chi_{orb}(\mathcal{M}) = \chi_{top}(F^{an})\chi_{orb}(\mathcal{N})\) where \(F\) is any fiber.\\

\noindent
Using this, we obtain:
\begin{align*}\chi_{orb}(\overline{\mathcal{J}}^d_{g,n}|_{\mathcal{M}^{\Gamma}}) = \chi_{orb}(\mathcal{M}^{\Gamma})\chi_{top}(\overline{J}^d(C)) \\=  \begin{cases}c(\Gamma)\chi_{orb}(\mathcal{M}^{\Gamma})\;\;\; \text{if }\Gamma \in G(g,n)^0\\
0\;\;\;\;\;\;\;\;\;\;\;\;\;\;\;\;\;\;\;\;\;\; \text{otherwise}.\end{cases}
\end{align*}
where \(C\) is any genus \(g\) nodal curve with dual graph \(\Gamma\).\\

\noindent
Using the fact that orbifold Euler characteristics are additive under locally closed stratifications, we obtain the formula:
\begin{align*}
\chi_{orb}(\overline{\mathcal{J}}^d_{g,n}) = \sum_{\Gamma \in G(g,n)^0}c(\Gamma)\chi_{orb}(\mathcal{M}^{\Gamma}) = \sum_{\Gamma \in G(g,n)^0}\sum\limits_{\substack{T \text{ is a}\\\text{spanning}\\ \text{tree of }\Gamma}}\chi_{orb}(\mathcal{M}^{\Gamma}).
\end{align*}
By combinatorial reasoning, we manipulate this double sum to obtain the formula:
\begin{align*}
\chi_{orb}(\overline{\mathcal{J}}^d_{g,n}) = \frac{1}{2^g(g!)}\chi_{orb}(\overline{\mathcal{M}}_{0,2g+n}).
\end{align*}
Since \(\overline{\mathcal{M}}_{0,2g+n}\) is a variety, \(\chi_{orb}(\overline{\mathcal{M}}_{0,2g+n}) = \chi(\overline{\mathcal{M}}_{0,2g+n})\).

\section{acknowledgements}
\noindent
The present paper forms part of a project at ETH Z\"urich, supervised by Rahul Pandharipande and Johannes Schmitt. We want to thank them for proposing the topic, very useful feedback and suggestions for improvements. We would also like to thank Nicola Pagani for suggesting the generalization of Theorem \ref{main theorem} to all fine compactified universal Jacobians, as studied in \cite{PT23} rather than considering just those defined in terms of a universal polarization as in \cite{KP19}.
\section{Definitions and notation}
\noindent
In the following section, we recall the theory of (compactified) Jacobians following the treatment of \cite{PT23}. Readers familiar with this subject are encouraged to skip to the next section and refer back to this section as needed.\\

\noindent
Associated to a complex, nonsingular projective curve \(C\) is a {\bf Jacobian \(J^0(C)\)} which is a complex, projective abelian variety whose dimension is equal to the genus of the curve. These are moduli spaces parametrizing degree \(0\) line bundles on \(C\) up to isomorphism. The group law on \(J^0(C)\) is by tensor product. In fact, as an abstract group, it is isomorphic to \((\mathbb{R}/\mathbb{Z})^{2g}\) and is the subgroup of \(\mathrm{Pic}(C)\) consisting of degree 0 line bundles. For \(g \geq 1\), it has \(n\)-torsion isomorphic to \((\mathbb{Z}/n\mathbb{Z})^{2g}\). In particular, for every \(n\), there is an element of order \(n\) in \(J^0(C)\) that generates a subgroup of order \(n\). If \(d \in \mathbb{Z}\), there is a variety \(J^d(C)\) parametrizing degree \(d\) line bundles which is isomorphic to \(J^0(C)\) as a variety where the isomorphism is given by tensoring with a fixed line bundle of degree \(d\). \(J^d(C)\) is no longer an algebraic group but \(J^0(C)\) acts freely on \(J^d(C)\) via tensor products.\\

\noindent
Similarly, singular complex projective curves \(C\) have associated {\bf generalized Jacobian varieties} \(J^{\underline{0}}(C)\) which in general fail to be proper over \(\mathbb{C}\) (see below for the definition). These Jacobian varieties are abelian group schemes and there are several approaches to compactifying these varieties.\\

\noindent
Given a nodal curve \(C\) and its total normalization \(v: C^v \rightarrow C\), if \(C_1,...,C_t\) are the irreducible components of \(C^v\), then the {\bf multidegree}  of \(\mathcal{L} \in \mathrm{Pic}(C)\) is the vector \(\underline{\mathrm{deg}(\mathcal{L})} = (\mathrm{deg}_{C_i}(v^*\mathcal{L}|_{C_i}))_i\). Its {\bf total degree}, \(\mathrm{deg}(\mathcal{L})\) is given by the sum of the components of the vector \(\underline{\mathrm{deg}(\mathcal{L})}\). As a group, \(J^{\underline{0}}(C)\) is the subgroup of \(Pic(C)\) consisting of line bundles of multidegree \(\underline{0}\) and for each \(\underline{d}\in \mathbb{Z}^t\), there is a scheme \(J^{\underline{d}}(C)\) parametrizing line bundles of multidegree \(\underline{d}\) on \(C\). In particular, these are all isomorphic to \(J^{\underline{0}}(C)\) where the isomorphism is given by tensoring with a fixed line bundle of multidegree \(\underline{d}\).\\

\noindent
Given a nodal curve \(C\) with dual graph \(\Gamma(C)\), we define:
\begin{align*}
b_1(\Gamma(C)) = |E(\Gamma(C))| - |V(\Gamma(C))| + 1.
\end{align*}
If \(\delta\) is the number of nodes of \(C\), then this quantity is \(\delta - t + 1\).\\

\noindent
There is an exact sequence of group schemes: (\cite[p.89]{ACVII})
\begin{align*}
1 \rightarrow (\mathbb{C}^{\times})^{b_1(\Gamma(C))} \rightarrow J^{\underline{0}}(C) \xrightarrow{\mathcal{L}\mapsto v^*(\mathcal{L})} J^{\underline{0}}(C^v) \cong \prod_{i=1}^tJ^0(C_i) \rightarrow 0 
\end{align*}
When  all the \(C_i\) are genus 0 curves, \(J^{\underline{d}}(C) \cong (\mathbb{C}^{\times})^{b_1(\Gamma(C))}\).\\

\noindent
The construction of Jacobians also works for families of smooth curves and for each \(g,n\) such that \(2g -2 + n>0\), there exists a universal Jacobian \(\mathcal{J}^d_{g,n}\) which is a DM stack along with a smooth, proper morphism \(\mathcal{J}^d_{g,n} \rightarrow \mathcal{M}_{g,n}\) such that the fiber over any geometric point \([(C;p_1,...,p_n)]\) of \(\mathcal{M}_{g,n}\) is the Jacobian variety \(J^d(C)\).\\

\noindent
A coherent sheaf on a nodal curve \(C\) has {\bf rank 1} if the stalk at each generic point of \(C\) has length 1. It is {\bf torsion free} if it has no embedded components and it is {\bf singular at P} if it fails to be locally free at \(P\). If \(\mathcal{F}\) is a torsion free sheaf on \(C\), we say it is {\bf simple} if its automorphism group is \(\mathbb{C}^{\times}\). This is equivalent to the condition that removing the singular points of \(\mathcal{F}\) from \(C\) does not disconnect \(C\).\\

\noindent
Every rank 1, torsion free, simple sheaf \(\mathcal{F}\) on a nodal curve \(C\)  can be associated with a pair \((S(\mathcal{F}), \underline{\mathrm{deg}(\mathcal{F})})\). Here \(S(\mathcal{F})\) is the set of nodes of \(C\) at which \(\mathcal{F}\) fails to be locally free. Since \(\mathcal{F}\) is simple, \(S(\mathcal{F})\) is a {\bf non-disconnecting} subset of nodes of \(C\). The multidegree \(\underline{\mathrm{deg}(\mathcal{F})}\) is the multidegree of the maximal torsion free quotient of the pullback of \(\mathcal{F}\) under the total normalization of \(C\). In other words, the partial normalization \(C_S\) of \(C\) at the nodes \(S(\mathcal{F})\) of \(C\) is connected.\\

\noindent
If \(C\) is a nodal curve, a degree \(d\) {\bf fine compactified Jacobian} \(\overline{J}^d(C)\) is a nonempty, connected, proper, open subscheme of \(Simp^d(C)\) where \(Simp^d(C)\) is the scheme parametrizing rank 1, torsion free, simple coherent sheaves on \(C\). It is {\bf smoothable} if there exists a regular smoothing \(\mathcal{C}\rightarrow Spec(\mathbb{C}[[t]])\) of \(C\) such that the fiber over \(0 \in \mathbb{C}[[t]]\) of a morphism \(U \rightarrow Spec(\mathbb{C}[[t]])\) where \(U\) is an open, proper subscheme of \(Simp^d(\mathcal{C}/\mathbb{C}[[t]])\) is \(\overline{J}^d(C)\).\\

\noindent
A degree \(d\)  {\bf fine compactified universal Jacobian} is an open substack of \(Simp^d(\overline{C}_{g,n}/\overline{\mathcal{M}}_{g,n})\) (see \cite{PT23} for a definition) that is proper over \(\overline{\mathcal{M}}_{g,n}\). In particular, since \(Simp^d(\overline{C}_{g,n}/\overline{\mathcal{M}}_{g,n})\) is representable over \(\overline{\mathcal{M}}_{g,n}\) and since open immersions of DM stacks are representable, any fine compactified universal Jacobian is representable over \(\overline{\mathcal{M}}_{g,n}\).\\

\noindent
We now define stability conditions and explain their relation to fine, smoothable compactified Jacobians. If \(G \subset \Gamma\) is a connected, spanning subgraph of \(\Gamma\), define
\begin{align*}
S^d_{\Gamma}(G) = \left \{ \underline{d}\in \mathbb{Z}^{|V(\Gamma)|}: \sum_{v \in V(\Gamma)}d(v) = d - |E(\Gamma)\setminus E(G)|\right \} \subset \mathbb{Z}^{|V(\Gamma)|}.
\end{align*}
If \(\mathcal{F}\) is a rank 1, torsion free, simple sheaf on a nodal curve \(C\), then \(\underline{\mathrm{deg}}(\mathcal{F}) \in S^d_{\Gamma}(\Gamma(\mathcal{F}))\) where \(\Gamma(\mathcal{F})\) is the graph obtained from \(\Gamma(C)\) by removing the edges corresponding to the nodes of \(C\) at which \(\mathcal{F}\) fails to be locally free.\\

\noindent
The {\bf twister} of a graph \(\Gamma\) at the vertex \(v\) is defined to be the element of \(\mathbb{Z}^{|V(\Gamma)|}\) defined by:
\begin{align*}
    \mathrm{Tw}_{\Gamma,v}(w) = \begin{cases}
        \substack{\# \text{ edges of }\Gamma\text{ having endpoints } v\text{ and } w,}\;\;\;\;\;\;\;\;\;v \neq w\\
        \substack{-\# \text{ non-loop edges of }\Gamma\\\text{ having one of its endpoints }v = w,}\;\;\;\;\;\;\;\;\;\;\;\;\;\;\;\;v = w
    \end{cases}.
\end{align*}
The twister group, \(\mathrm{Tw}(\Gamma)\) is the subgroup of \(\mathbb{Z}^{|V(\Gamma)|}\) generated by the elements \(\lbrace\mathrm{Tw}_{\Gamma,v}\rbrace_{v \in V(\Gamma)}\). We have the inclusions \(\mathrm{Tw}(\Gamma) \subset S^0_{\Gamma}(\Gamma) \subset \mathbb{Z}^{|V(\Gamma)|}\) and therefore, for a connected, spanning subgraph \(G \subset \Gamma\), \(\mathrm{Tw}(G)\) acts on \(S^d_{\Gamma}(G)\) by vector addition.\\

\noindent
A {\bf degree \(d\) stability condition \(\sigma\)} on \(\Gamma\) is a set of pairs \((G, \underline{d})\) such that \(\underline{d} \in S^d_{\Gamma}(G)\) where \(G \subset \Gamma\)  is a connected, spanning subgraph. The pairs \((G, \underline{d})\) are additionally required to satisfy the following two conditions:\\

\noindent
(1) If \(G\) is a connected, spanning subgraph of \(\Gamma\), then for all edges \(e\) of \(E(\Gamma)\setminus E(G)\) with endpoints \(v_1\) and \(v_2\), if \((G,\underline{d}) \in \sigma\), then \((G \cup \lbrace e \rbrace, \underline{d} + \underline{e}_{v_i})\) is in \(\sigma\) for \(i = 1,2\) where \(\underline{e}_{v_i}\) is the standard basis vector of \(\mathbb{Z}^{|V(\Gamma)|}\) corresponding to vertex \(v_i\).\\
(2) For every connected, spanning subgraph \(G\) of \(\Gamma\):
\begin{align*}
    \sigma(G) = \lbrace \underline{d}: (G, \underline{d}) \in \sigma \rbrace \subset S^d_{G}(G)
\end{align*}
is a minimal, complete set of representatives for the action of the twister group Tw(\(G\)) on \(S^d_{\Gamma}(G)\).\\

\noindent
For every connected, spanning subgraph \(G\) of \(\Gamma\), we have \(|\sigma(G)| = c(G)\) by \cite[Remark 4.4]{PT23}.\\

\noindent
For a nodal curve \(C\), the scheme \(Simp^d(C)\) has a stratification into locally closed subsets \(Simp^d(C) = \underset{(G, \underline{d})}{\bigsqcup} \overline{J}_{(G,\underline{d})}(C)\) where \(\overline{J}_{(G,\underline{d})}(C)\) is the locus whose points correspond to sheaves \(\mathcal{F}\) where \(\Gamma(\mathcal{F}) = G\) and \(\underline{\mathrm{deg}(\mathcal{F})} = \underline{d}\). Here the disjoint union is over all connected, spanning subgraphs \(G\) of \(\Gamma\) and all \(\underline{d}\in \mathbb{Z}^{|V(\Gamma)|}\) such that
\begin{align*}
\sum_{i = 1}^{|V(\Gamma)|}d_i = d - |E(\Gamma)\setminus E(G)|.
\end{align*}
We view these as schemes, endowed with the reduced schematic structure.\\

\noindent
Given a stability condition \(\sigma\) on the dual graph \(\Gamma\) of a nodal curve \(C\), a rank 1, torsion free simple sheaf \(\mathcal{F}\) is {\bf \(\sigma\)-stable} if \((\Gamma(\mathcal{F}), \underline{\mathrm{deg}(\mathcal{F}}))\) is in \(\sigma\) where \(\Gamma(\mathcal{F})\) is the subgraph of \(\Gamma(C)\) obtained by removing edges of \(\Gamma(C)\) corresponding to nodes of \(C\) at which \(\mathcal{F}\) fails to be locally free.

\begin{theorem}\cite[Corollary 6.4]{PT23}
    Given a degree \(d\) stability condition \(\sigma\) on the dual graph \(\Gamma(C)\) of a nodal curve \(C\), the subscheme \(\overline{J}^d_{\sigma}(C)\) of \(Simp^d(X)\) parametrizing sheaves that are \(\sigma\)-stable is a smoothable, degree \(d\) fine compactified Jacobian for \(C\).
    \end{theorem}

\begin{theorem}\cite[Corollary 7.11]{PT23}\label{bijection}
If \(C\) is a nodal curve, then there is a bijection between degree \(d\) fine, smoothable compactified Jacobians of \(C\) and degree \(d\) stability conditions on \(\Gamma(C)\).
\end{theorem}
\noindent
This will play an important role in our calculation of the orbifold Euler characteristics of fine compactified universal Jacobians. To be precise, for every fine, smoothable compactified Jacobian \(\overline{J}^d(C)\) of \(C\), there is some unique stability condition \(\sigma\) on \(\Gamma(C)\) such that \(\overline{J}^d(C) = \overline{J}_{\sigma}^d(C)\).\\

\noindent
A non-degenerate {\bf polarization} \(\phi\) (see \cite{PT23}) on a stable graph \(\Gamma\) defines a stability condition \(\sigma_{\phi}\) on \(\Gamma\). We also have the notion of a non-degenerate universal polarization, which is a collection \(\phi = (\phi_{\Gamma})_{\Gamma \in G(g,n)}\) such that each \(\phi_{\Gamma}\) is a non-degenerate polarization on \(\Gamma\) and such that the \(\phi_{\Gamma}\) are compatible with graph morphisms. Every non-degenerate polarization on a graph \(\Gamma\) gives rise to a smoothable, fine compactified Jacobian on curves with dual graph \(\Gamma\). Additionally, non-degenerate universal polarizations give rise to fine compactified universal Jacobians.\\

\noindent
The fine compactified universal Jacobians \(\overline{\mathcal{J}}^d_{g,n}(\phi)\) over \(\overline{\mathcal{M}}_{g,n}\) corresponding to a universal polarization \(\phi = (\phi_{\Gamma})_{\Gamma \in G(g,n)}\) are introduced in \cite{KP19}. 

\begin{theorem}\cite[Proposition 2.9]{PT22}\\
For every non-degenerate universal polarization \(\phi = (\phi_{\Gamma})_{\Gamma \in G(g,n)}\), the moduli stack \(\overline{\mathcal{J}}^d_{g,n}(\phi)\) is a fine compactified universal Jacobian.
\end{theorem}

\noindent
There are fine compactified universal Jacobians that do not arise from universal polarizations. For example, in \cite[Section 6]{PT22}, fine compactified universal Jacobians \(\overline{\mathcal{J}}^d_{g,n}\), not arising from a universal stability condition are found for \(n \geq 6\). In fact, it is rarely the case that all fine compactified Jacobians arise from universal stability conditions.

\section{Orbifold Euler characteristics of separated, finite type, complex DM stacks}
\noindent
In this section, we recall basic properties of the orbifold Euler characteristics of finite type, separated Deligne--Mumford (DM) stacks defined over \(\mathbb{C}\) and prove Proposition \ref{constant fiber multiplicativity} which is central in our calculation.

\noindent
\begin{theorem}\label{three properties}
The orbifold Euler characteristic \(\chi_{orb}(\mathcal{M})\) of a finite type separated DM stack \(\mathcal{M}\), satisfies and is characterized by the following three properties:
\begin{enumerate}
\item[1.] If \(\bigsqcup_{i=1}^n{\mathcal{M}}_i = \mathcal{M}\) is a stratification of \(\mathcal{M}\) into finitely many locally closed substacks, then \begin{align*}\chi_{orb}(\mathcal{M}) = \sum_{i=1}^n\chi_{orb}(\mathcal{M}_i).\end{align*}
\item[2.]If \(f: \mathcal{M} \rightarrow \mathcal{N}\) is a finite, surjective, \'etale morphism of finite type separated DM stacks where \(\mathcal{N}\) is integral, then \begin{align*}\chi_{orb}(\mathcal{M}) = \mathrm{deg}(f)\chi_{orb}(\mathcal{N})\end{align*}where \(\mathrm{deg}(f)\) is defined as in \cite{Vis}.
\item[3.]If \(\mathcal{M}\) is a scheme, then \begin{align*}\chi_{orb}(\mathcal{M}) = \chi_{top}(\mathcal{M}^{an})\end{align*}
where \(\mathcal{M}^{an}\) is the scheme \(\mathcal{M}\) viewed as a topological space endowed with the complex topology from \cite[Appendix B]{Har}.
\end{enumerate}
\end{theorem}
\noindent
The existence of an orbifold Euler characteristic satisfying these properties is well known within the mathematical community. Once existence of such a definition has been shown, uniqueness follows from its properties.

\begin{remark}
In this paper, when we speak of locally closed substacks of \(\mathcal{M}\), we mean locally closed subsets with respect to the underlying topology of \(\mathcal{M}\) with the unique reduced structure. Moreover, the compliment of a locally closed substack is a locally closed substack in the same way.
\end{remark}

\begin{theorem} \cite[ Theorem 4.5.1]{Alp} \label{finite morphism from a scheme}\\
If \(\mathcal{M}\) is a separated, finite type DM stack over a noetherian scheme, there exists a finite, surjective, generically \'etale  morphism \(Z \rightarrow \mathcal{M}\) from a scheme \(Z\).
\end{theorem}

\begin{lemma}\label{stratification lemma}
A DM stack \(\mathcal{M}\) which is finite type and separated over \(\mathbb{C}\) has a stratification by finitely many locally closed substacks \(\mathcal{U}_i\) such that \(\mathcal{M} = \bigsqcup_{i=0}^n\mathcal{U}_i\) with each of the \(\mathcal{U}_i\) covered by a finite, surjective, \'etale morphism from a scheme.
\end{lemma}
\begin{proof}
    Separated DM stacks, of finite type over \(\mathbb{C}\) are noetherian. Therefore, we may argue by noetherian induction. There is some atlas \(U \rightarrow \mathcal{M}\) where \(U\) is a reduced, separated, finite type scheme over \(\mathbb{C}\). This scheme \(U\) contains a dense, smooth open subscheme \(U_0\). The image of \(U_0\) under \(U \rightarrow \mathcal{M}\) is a nonempty open, smooth substack \(\mathcal{U}_0\) of \(\mathcal{M}\). Then applying Lemma \ref{stratification lemma}, shrinking \(\mathcal{U}_0\) if necessary, we may assume \(\mathcal{U}_0\) is smooth and that there is some finite, surjective, \'etale morphism from a scheme \(U\) to \(\mathcal{U}_0\) by Theorem \ref{finite morphism from a scheme}. We may further assume that \(\mathcal{U}_0\) is integral by taking a connected component.\\

    \noindent
    Let \(\mathcal{V}_0 = \mathcal{M}\setminus\mathcal{U}_0\). If \(\mathcal{V}_{0} \neq \emptyset\), repeat the above process to find a nonempty open substack \(\mathcal{U}_1\) of \(\mathcal{V}_0\) with the property that \(\mathcal{U}_1\) has a finite, surjective, \'etale cover by a scheme. Set \(\mathcal{V}_1 = \mathcal{V}_0\setminus\mathcal{U}_0\). Continue inductively until \(\mathcal{V}_i = \emptyset\). This process must terminate after finitely many iterations since we obtain a chain of closed subsets \(\mathcal{V}_0 \supsetneq \mathcal{V}_1 \supsetneq ...\) of \(\mathcal{M}\). Therefore, using the fact that \(\mathcal{M}\) is noetherian, there must be some finite \(n\) such that \(\mathcal{V}_n = \emptyset\). We obtain \(\mathcal{M} = \bigsqcup_{i=0}^n\mathcal{U}_i\) where the \(\mathcal{U}_i\) are locally closed, integral substacks of \(\mathcal{M}\), each having a finite, surjective \'etale morphism from a scheme.
    \end{proof}

\begin{remark}
    The smoothness at the beginning of the proof of Lemma \ref{stratification lemma} was only introduced so that we could simultaneously obtain openess and integrality of \(\mathcal{U}_0\) in \(\mathcal{M}\). The openess is needed to proceed by noetherian induction. In general, an open subset of an integral component is not open. However, for smooth DM stacks, integral components are connected components, so are open.
\end{remark}
\begin{remark}
    An explicit construction of orbifold Euler characteristics for finite type, separated DM stacks over \(\mathbb{C}\) can be defined as follows. Take a locally closed stratification of \(\mathcal{M} = \bigsqcup_{i = 1}^n\mathcal{U}_i\) as in Lemma \ref{stratification lemma} where there is a finite, surjective, \'etale morphism \(U_i \xrightarrow{f_i} \mathcal{U}_i\) from some scheme \(U_i\) to \(\mathcal{U}_i\) for each \(i\). We can define
    \begin{align*}
    \chi_{orb}(\mathcal{M}) = \sum_{i = 1}^n\frac{\chi_{top}(U_i^{an})}{\mathrm{deg}(f_i)}.
    \end{align*}
    This can be shown to be independent of the choice of stratification and of the \(f_i\) once such a stratification has been chosen. The properties in Theorem \ref{three properties} can easily be seen to follow from this definition.
\end{remark}

\begin{theorem}\label{top fibration theorem} \cite[Theorem 1, p 481]{Es}\\
If \(p: E \rightarrow B\) is an orientable topological fibration with fiber \(F\) and \(B\) is path-connected, then \(\chi_{top}(E) = \chi_{top}(F)\chi_{top}(B)\).\end{theorem}

\begin{theorem}\cite[Section 1.2]{Zei}\label{proper, surjective implies fibration}
Given a proper, surjective morphism of schemes \(X \rightarrow Y\). There is a finite, locally closed stratification \(Y = \bigsqcup_{i=1}^m Y_i\) of the base such that \(X|_{Y_i} = X\times_YY_i \rightarrow Y_i\) is a fibration with respect to the complex topology.
\end{theorem}

\begin{proposition}\label{constant fiber multiplicativity}
Suppose \(\mathcal{M} \rightarrow \mathcal{N}\) is a proper, surjective, representable morphism of separated, finite type DM stacks such that the topological Euler characteristic of the fiber of any geometric point  of \(\mathcal{N}\) is independent of the choice of geometric point. Then
\begin{align*}
\chi_{orb}(\mathcal{M}) = \chi_{top}(F^{an})\chi_{orb}(\mathcal{N})
\end{align*}
where \(F\) is the fiber over any geometric point of \(\mathcal{N}\).
\end{proposition}

\begin{proof}
By property \(1\) of Theorem \ref{three properties} and Lemma \ref{stratification lemma}, it suffices to consider the case where \(\mathcal{N}\) is integral and has a finite, surjective, \'etale morphism \(U \xrightarrow{f}\mathcal{N}\) from a scheme.\\

\noindent
We obtain the following cartesian diagram
\[
\begin{tikzcd}
\mathcal{M}\times_{\mathcal{N}} U \arrow[r] \arrow[d,swap,"g"] &
 U \arrow[d, "f"] \\
\mathcal{M} \arrow[r] & \mathcal{N}
\end{tikzcd}
\]
We take a locally closed stratification \(U = \bigsqcup_{j=1}^n U_j\) as in Theorem \ref{proper, surjective implies fibration}. Taking connected components, we may further assume that they are connected. Letting \(F_{j}\) be a fiber of \(\mathcal{M}|_{U_j} \rightarrow U_j\) by Lemma \ref{proper, surjective implies fibration}, we obtain
\begin{align*}
    \chi_{top}((\mathcal{M}\times_{\mathcal{N}}U|_{U_j})^{an}) = \chi_{top}(F_{j}^{an})\chi_{top}(U_j^{an}).
\end{align*}
Using the fact that \(F_{j}\) is a fiber of \(\mathcal{M} \rightarrow \mathcal{N}\) and writing \(F\) for any general fiber of this morphism,
\begin{align*}
    \chi_{top}(F^{an}) = \chi_{top}(F_{j}^{an}).
\end{align*}
By Lemma \ref{stratification lemma}, we can obtain a locally closed stratification of \(\mathcal{M}\) into finitely many locally closed, integral substacks. Using the fact that over each integral substack \(\mathcal{M}' \subset \mathcal{M}\), \(g|_{\mathcal{M}'}\) is a finite, surjective \'etale morphism of degree \(\mathrm{deg}(f)\), by properties 1 and 3 of Theorem \ref{three properties},
\begin{align*}
\chi_{orb}(\mathcal{M}) = \frac{\chi_{orb}(\mathcal{M}\times_{\mathcal{N}}U)}{\mathrm{deg}(f)} = \frac{\chi_{top}((\mathcal{M}\times_{\mathcal{N}}U)^{an})}{\mathrm{deg}(f)}. 
\end{align*}
We compute
\begin{align*}
    \chi_{orb}(\mathcal{M}) = \frac{\chi_{top}((\mathcal{M}\times_{\mathcal{N}}U)^{an})}{\mathrm{deg}(f)}\\=\frac{1}{\mathrm{deg}(f)}\sum_{j=1}^n\chi_{top}((\mathcal{M}\times_{\mathcal{N}}U|_{U_j})^{an})\\
    = \frac{1}{\mathrm{deg}(f)}\sum_{j=1}^n\chi_{top}(U_j^{an})\chi_{top}(F^{an}) \\= \chi_{top}(F^{an})\frac{\chi_{top}(U^{an})}{\mathrm{deg}(f)} = \chi_{top}(F^{an})\chi_{orb}(\mathcal{N}).
\end{align*}
\end{proof}

\section{Proof of the main result}
\noindent
If \(\overline{\mathcal{J}}^d_{g,n}\) is a fine compactified universal Jacobian over \(\overline{\mathcal{M}}_{g,n}\), then the fiber \(\overline{J}^d(C)\) over any point \([(C;p_1,...,p_n)]\) of \(\overline{\mathcal{M}}_{g,n}\) is a degree \(d\), fine smoothable compactified Jacobian of \(C\). The smoothability of the fibers \(\overline{J}^d(C)\) follows from the fact that the morphism \(\overline{\mathcal{J}}^d_{g,n}\rightarrow \overline{\mathcal{M}}_{g,n}\) is smooth over the open, dense substack \(\mathcal{M}_{g,n}\) of \(\overline{\mathcal{M}}_{g,n}\). Therefore, every fiber \(\overline{J}^d(C)\) over a point \([(C;p_1,...,p_n)]\) is of the form \(\overline{J}^d(C) = \overline{J}^d_{\sigma}(C)\) for some stability condition \(\sigma\) on \(\Gamma(C)\).\\

\noindent
\begin{lemma}\cite[Lemma 1.5]{Alex}\label{line bundle lemma 1}
A rank 1, torsion free, simple sheaf \(\mathcal{F}\) of a nodal curve \(C\) is the direct image \(v_*\mathcal{F}'\) under some partial normalization \(v: C' \rightarrow C\) at the set of nodes where \(\mathcal{F}\) fails to be locally free for some unique line bundle \(\mathcal{F}'\) on \(C'\).
\end{lemma}

\noindent
Let \(C_S\) be the partial normalization of \(C\) at the subset of nodes corresponding to edge set \(S\) in \(E(\Gamma(C))\). For any line bundle \(\mathcal{L}\) on \(C_S\), by \cite{MeVi} Proposition 1.14 (iii), if \(S_v\) is the set of self-edges in \(S\) incident at vertex \(v\) and \(C^v\) is the corresponding irreducible component of \(C\)
\begin{align}\label{pushforward equation}
\mathrm{deg}_{C^v}(\mathcal{L}|_{C^v_S}) = \mathrm{deg}_{C^v}((v_S)_*\mathcal{L}|_{C^v}) - |S_v|.
\end{align}

\noindent
As in \cite{MeVi}, for a nodal curve \(C\), given a fine, smoothable compactified Jacobian \(\overline{J}^d_{\sigma}(C)\) corresponding to a stability condition \(\sigma\), we have the stratification \(\overline{J}^d_{\sigma}(C) = \bigsqcup_{(G, \underline{d}) \in \sigma}\overline{J}_{(G,\underline{d})}(C)\) where \(\overline{J}_{(G,\underline{d})}(C)\) is the locus whose points correspond to sheaves \(\mathcal{F}\) where \(\Gamma(\mathcal{F}) = G\) and \(\underline{\mathrm{deg}(\mathcal{F})} = \underline{d}\). These are locally closed subsets which we view as schemes endowed with the reduced schematic structure.\\

\noindent
\begin{lemma}\label{stability pushforward}
    If \(\sigma\) is a stability condition on a graph \(\Gamma\) and \(S\) is a set of edges whose removal does not result in a disconnected graph, and if \(\Gamma_S\) is the resulting connected graph, then we can define a stability condition \(\sigma_S\) on \(\Gamma_S\) to be:
    \begin{align*}
\sigma_S = \left\{\begin{array}{lr}
             (G, (d_v - |S_v|)_{v \in V(\Gamma)}): G \subset \Gamma_S \text{ is spanning and connected},\\ \;\;\;\;\;\;\;\;\;\;\;\;\;\;\;\;\;\;\;\;\;\;\;\;\;\;\;\;\;\;\;\;\;\;\;(d_v)_{v \in V(\Gamma)} \in \sigma(G)
        \end{array}\right\}.
    \end{align*}
Moreover, if \(\sigma\) is a stability condition on \(\Gamma\) of degree \(d\), then \(\sigma_S\) is a degree
\begin{align*}
d_S = d + (|S| - \underset{v\in V(\Gamma)}{\sum}|S_v|),
\end{align*}
or alternatively,
\begin{align*}
d_S = d + \#\lbrace\text{edges in }S \text{ that are not self-edges}\rbrace.
\end{align*}
\end{lemma}
\begin{proof}
We have:
\begin{align*}
    \sigma_S
    \subset \lbrace\text{connected, spanning subgraphs of }\Gamma_S\rbrace\times\mathbb{Z}^{|V(\Gamma_S)|}.
\end{align*}
So it suffices to check that the two conditions in the definition of a degree \(d_S\) stability condition hold.\\

\noindent
(1) If \(G\) is a connected, spanning subgraph of \(\Gamma_S\), then if \(e\) is an edge in \(E(\Gamma_S) \setminus E(G)\), then for \((d_v)_v \in \sigma_S(G)\), we have \((d_v + |S_v|)_v \in \sigma(G)\) and therefore if \(v_1\) and \(v_2\) are the endpoints of \(e\), \(\underline{e}_{v_i} + (d_v + |S_v|)_v \) is in \(\sigma(G \cup \lbrace e \rbrace)\) and \(\underline{e}_{v_i} + (d_v )_v \) is in \(\sigma_S(G\cup \lbrace e \rbrace)\) for \(i = 1,2\).\\

\noindent
(2) For each connected, spanning subgraph \(G\) of \(\Gamma_S\) and for \(\underline{d} \in \sigma_S(G)\), we have:
\begin{align*}
    \underset{v \in V(\Gamma)}{\sum}d_v = d - \underset{v \in V(\Gamma)}{\sum}|S_v| - |E(\Gamma)\setminus E(G)| \\= d + (|S| - \underset{v \in V(\Gamma)}{\sum}|S_v|) - |E(\Gamma_S)\setminus E(G)|\\ = d_S - |E(\Gamma_S)\setminus E(G)|.
\end{align*}
Therefore \(\sigma_S(G) \subset S^{d_S}_{\Gamma_S}(G)\). We also have \(S^{d_S}_{\Gamma_S}(G) = S^{d}_{\Gamma}(G) - (|S_v|)_v\) and therefore, since \(\sigma(G)\) is a minimal, complete set of representatives for the action of Tw(\(G\)) on \(S^{d_S}_{\Gamma_S}(G)\), using the fact that \(\sigma_S(G) = \sigma(G) - (|S_v|)_v\).

\end{proof}

\noindent
Given a stability condition \(\sigma\) on \(\Gamma = \Gamma(C)\), we write:
\begin{align*}
    \overline{J}_{\sigma,\Gamma_S}(C) = \bigsqcup_{\underline{d} \in \sigma(\Gamma_S)}\overline{J}_{(\Gamma_S,\underline{d})}(C) \subset \overline{J}^d_{\sigma}(C).
\end{align*}

\noindent
The following proof is adapted from the proof of \cite[Theorem 4.1]{MeVi}.

\begin{theorem}\label{Jacobian stratification}
Given a nodal curve and a degree \(d\) stability condition, \(\sigma\) on  \(\Gamma = \Gamma(C)\), there is a morphism \((v_S)_*: Simp^{d_S}(C_S) \rightarrow Simp^{d}(C)\) induced by the normalization \(v_S: C_S \rightarrow C\). It gives rise to an isomorphism
\begin{align*}
     (\overline{J}^{d_S}_{\sigma_{S}}(C_S))_{sm} = \overline{J}^{d_S}_{\sigma_S,\Gamma_S}(C_S) \cong \overline{J}^d_{\sigma, \Gamma}(C).
\end{align*}
Here \((\overline{J}^{d_S}_{\sigma_{S}}(C_S))_{sm}\) denotes the smooth locus of \(\overline{J}^{d_S}_{\sigma_{S}}(C_S)\), or alternatively, the locus where points correspond to line bundles. 
\end{theorem}

\begin{proof}
The morphism \((v_S)_*: Simp^{d_S}(C_S) \rightarrow Simp^{d}(C)\) is a monomorphism as a result of \cite[Lemma 3.4]{EsKl}, which may be applied with families of reduced, not just integral curves (see proof of \cite[Theorem 4.1]{MeVi}). It states that for a given scheme \(T\), the functor
\begin{align*}
    \left\{\begin{array}{lr}T\text{-flat rank 1, torsion free}\\ \text{sheaves on } C_S\times T \text{ which}\\
    \text{have degree }d\text{ over fibers}\\\text{of } C_S\times T \rightarrow T\end{array}\right\} \xrightarrow{(v_{S,T})_*} \left\{\begin{array}{lr}T\text{-flat rank 1, torsion free}\\ \text{sheaves on } C\times T \text{ which}\\
    \text{have degree }d\text{ over fibers}\\\text{of } C\times T \rightarrow T\end{array}\right\}
\end{align*}
where \(v_{S,T}\) is the morphism \(v_{S,T}:C_S\times T \rightarrow C \times T\), is a fully faithful embedding. By Equation \ref{pushforward equation} and Lemma \ref{stability pushforward}, this induces a morphism:
\begin{align*}
(v_S)_*: \overline{J}^d_{\sigma_S}(C_S) \rightarrow \bigsqcup_{G \subset \Gamma_S}\overline{J}_{\sigma,G}(C)\subset \overline{J}^d_{\sigma}(C)
\end{align*}
 over \(\bigsqcup_{G \subset \Gamma_S}\overline{J}_{\sigma,G}(C)\). This is a morphism from a proper scheme to a separated scheme and is therefore a proper monomorphism so is a closed immersion. By Lemma \ref{line bundle lemma 1} and Equation \ref{pushforward equation}, it induces a bijection of geometric points on the restriction \((\overline{J}^{d_S}_{\sigma_{S}}(C_S))_{sm}\rightarrow \overline{J}^d_{\sigma, \Gamma_S}(C)\) over \(\overline{J}^d_{\sigma, \Gamma_S}(C)\). Therefore, the morphism  \((\overline{J}^{d_S}_{\sigma_{S}}(C_S))_{sm} \rightarrow \overline{J}^d_{\sigma, \Gamma_S}(C)\) is an isomorphism.
\end{proof}

\begin{remark}
For each \(\underline{d} \in \sigma_S(\Gamma_S)\), we have \(\overline{J}_{(\Gamma_S,\underline{d})} = J^{\underline{d}}(C_S)\), the degree \(\underline{d}\) generalized Jacobian of \(C_S\). Therefore, Theorem \ref{Jacobian stratification} tells us that, given a stability condition \(\sigma\) on \(\Gamma(C)\), \(\overline{J}^d_{\sigma}(C)\) is the disjoint union of locally closed strata, each of which is isomorphic to a generalized Jacobian of some partial normalization of \(C\) at a subset of non-disconnecting nodes.
\end{remark}

\begin{corollary}\label{Jacobian stratification 2}
Suppose we have a fine, smoothable compactified Jacobian \(\overline{J}^d_{\sigma}(C)\) of a nodal curve \(C\) corresponding to a stability condition \(\sigma\) on \(\Gamma = \Gamma(C)\) and a subset \(S \subset E(\Gamma(C))\) whose removal results in a connected graph. Then \(\overline{J}_{\sigma,\Gamma_S}(C)\) has a stratification into \(|\sigma_S(\Gamma_S)| = c(\Gamma_S)\) strata, each of which is isomorphic to a generalized Jacobian on \(C_S\). By abuse of notation, we write
\begin{align*}
\overline{J}_{\sigma,\Gamma_S}(C) = \underset{\underline{d} \in \sigma_S(\Gamma_S)}{\bigsqcup}J^{\underline{d}}(C_S),
\end{align*}
and
\begin{align*}
\overline{J}^d_{\sigma}(C) = \underset{\substack{S \subset E(\Gamma)\\ \text{non-disconnecting}}}{\bigsqcup}\underset{\underline{d} \in \sigma_S(\Gamma_S)}{\bigsqcup}J^{\underline{d}}(C_S),
\end{align*}
where \(J^{\underline{d}}(C_S)\) is the generalized Jacobian of \(C_S\) of multidegree \(\underline{d}\).
\end{corollary}

\noindent
Let \(G(g,n)^0\subset G(g,n)\) be the set of representatives of isomorphism classes of stable graphs where all vertices have genus 0.
We recall the notation that if \(C\) is a nodal curve: 
\begin{align*}
b_1(\Gamma(C)) = |E(\Gamma(C))| - |V(\Gamma(C))| + 1 \geq 0.
\end{align*}

\begin{lemma}\label{nonzero genus}
If \(C\) is any nodal curve such that in its total normalization, there is some connected component which has genus at least 1, or equivalently if \(\Gamma(C) \in G(g,n)\setminus G(g,n)^0\), then for any multidegree \(\underline{d}\in \mathbb{Z}^{|V(\Gamma(C))|}\), it follows that \(\chi_{top}((J^{\underline{d}}(C))^{an}) = 0\).
\end{lemma}

\noindent
This follows from \cite[Proposition 2.2]{Beau}

\begin{corollary}\label{nonzero genus 1}
If \(C\) is a nodal curve whose dual graph \(\Gamma(C)\) lies in \(G(g,n)\setminus G(g,n)^0\), then for any stability condition \(\sigma\) on \(\Gamma(C)\), we have \(\chi_{top}((\overline{J}^d_{\sigma}(C))^{an})\)\( = 0\).
\end{corollary}
\begin{proof}
This follows from the fact that for finite type separated schemes, the topological Euler characteristic is additive with respect to locally closed stratifications, the stratification given in Theorem \ref{Jacobian stratification} and Lemma \ref{nonzero genus}. For each non-disconnecting subset \(S \subset E(C)\), the partial normalization \(C_{S}\) has dual graph with vertices having the same genera as those of \(\Gamma = \Gamma(C)\). So for each \(\underline{d} \in \sigma_{S}\), by Lemma \ref{nonzero genus}, \(\chi_{top}((J^{\underline{d}}(C_S))^{an}) = 0\) and therefore:
\begin{align*}
\chi_{top}((\overline{J}^d_{\sigma}(C))^{an}) = \sum\limits_{\substack{S \subset E(\Gamma)\\ \text{non-disconnecting}}}\chi_{top}((\overline{J}_{\sigma, \Gamma_S}(C))^{an})\\
     = \sum\limits_{\substack{S \subset E(\Gamma)\\ \text{ non-disconnecting}}}\sum_{\underline{d} \in \sigma_{S}}\chi_{top}((J^{\underline{d}}(C_{S}))^{an}) = 0. \;\; \;\;\;\;\qedhere
\end{align*}
\end{proof}

\begin{lemma}\label{rational strata oec}
Let \(C\) be a nodal curve such that \(\Gamma(C) \in G(g,n)^0\), let \(\underline{d} \in \mathbb{Z}^{|V(\Gamma(C))|}\) be any multidegree and \(S \subset E(\Gamma(C))\), then:
\begin{align*}
    \chi_{top}((J^{\underline{d}}(C_{S}))^{an}) = \begin{cases}1\;\;\;\;\; \Gamma(C_S)\text{ is a tree},\\
    0\;\;\;\;\text{ otherwise.}
\end{cases}
\end{align*}
\end{lemma}
\begin{proof}
We have an isomorphism \(J^{\underline{d}}(C_{S}) \cong J^{\underline{0}}(C_{S})\) given by tensoring with a fixed line bundle on \(C_{S}\) of appropriate multidegree. By  \cite[p.89]{ACVII}, \(J^{\underline{0}}(C) \cong (\mathbb{C}^{\times})^{b_1(\Gamma(C))}\). For \(n>0\), \(\chi_{top}((\mathbb{C}^{\times})^n) = 0\). So all that remains is to consider the case when \(|E(C_S)| = |V(C_S)|-1\), or equivalently, when \(b_1(\Gamma(C_S)) = 0\). This corresponds to \(C_S\) being a tree and, in this case, \(J^{\underline{0}}(C_{S})\) is a point so \(\chi_{top}((J^{\underline{0}}(C_{S}))^{an}) = 1\).
\end{proof}

\begin{corollary}\label{oec of rational strata}
If \([(C,p_1,...,p_n)]\) is a nodal curve of genus 0 such that the corresponding stable graph \(\Gamma = \Gamma(C)\) lies in \(G(g,n)^0\), then:
\begin{align*}
    \chi_{top}((\overline{J}^d_{\sigma}(C))^{an}) = c(\Gamma)
\end{align*}
for any stability condition \(\sigma\) on \(\Gamma\).
\end{corollary}
\begin{proof}
By the additivity of the orbifold Euler characteristic with respect to locally closed stratifications:
\begin{align*}
\chi_{top}(\overline{J}^d_{\sigma}(C)) = \sum_{S \subset E(\Gamma)}\chi_{top}(\overline{J}_{\sigma, \Gamma_S}(C))\\ = \sum_{S \subset E(\Gamma)}\sum_{\underline{d} \in \sigma_S(\Gamma_S)}\chi_{orb}(J^{\underline{d}}(C_{S}))\\ = \sum_{\lbrace S \subset E(\Gamma): \Gamma_S \text{ is a tree}\rbrace}|\sigma_S(\Gamma_S)|.
\end{align*}
The first equality follows from the stratification \(\overline{J}^d_{\sigma}(C) = \underset{S \subset E(\Gamma)}{\bigsqcup}\overline{J}_{\sigma, 
\Gamma_S}(C)\). The second follows from the stratification given in Corollary \ref{Jacobian stratification 2} and the third follows from Lemma \ref{oec of rational strata}.\\

\noindent
Since \(|\sigma_S(\Gamma_S)| = c(\Gamma_S) = 1\), whenever \(\Gamma_S\) is a tree, it follows that \(\chi_{top}(\overline{J}^d_{\sigma}(C)) = c(\Gamma)\).
\end{proof}

\begin{theorem}\label{compactified jacobian oec formula}
For all \(d \in \mathbb{Z}\), \(g,n \geq 0\) such that \(2g-2+n>0\) and all fine compactified universal Jacobians \(\overline{\mathcal{J}}^d_{g,n}\) of degree \(d\),
\begin{align*}
\chi_{orb}(\overline{\mathcal{J}}^d_{g,n}) = \sum_{\Gamma \in G(g,n)^0}c(\Gamma)\chi_{orb}(\mathcal{M}^{\Gamma}).\end{align*}
In particular, this is independent of \(d\) and the choice of fine compactified universal Jacobian.
\end{theorem}
\begin{proof}
\noindent
The stratification of \(\overline{\mathcal{M}}_{g,n}\) by dual graph induces a stratification of \(\overline{\mathcal{J}}^d_{g,n} = \underset{\Gamma \in G(g,n)}{\bigsqcup}\overline{\mathcal{J}}^d_{g,n}|_{\mathcal{M}^{\Gamma}}\). By additivity of orbifold Euler characteristics with respect to locally closed stratifications:
\begin{align*}
\chi_{orb}(\overline{\mathcal{J}}^d_{g,n}) = \sum_{\Gamma \in G(g,n)}\chi_{orb}(\overline{\mathcal{J}}^d_{g,n}|_{\mathcal{M}^{\Gamma}}).
\end{align*}
Suppose \(\Gamma \in G(g,n)\) and \([(C^{\Gamma};p_1,...,p_n)]\) is in \(\mathcal{M}^{\Gamma}(\mathbb{C})\) such that \\\(\overline{J}^d_{\sigma(C^{\Gamma})}(C^{\Gamma})\) is the fiber over \([(C^{\Gamma};p_1,...,p_n)]\) which is a smoothable, fine compactified Jacobian with \(\sigma(C^{\Gamma})\) being a corresponding stability condition. Then combining Proposition \ref{constant fiber multiplicativity} with Corollary \ref{nonzero genus 1} and Corollary \ref{oec of rational strata}, we obtain:
\begin{align*}
\chi_{orb}(\overline{\mathcal{J}}^d_{g,n}|_{\mathcal{M}^{\Gamma}}) = \chi_{orb}(\mathcal{M}^{\Gamma})\chi_{top}(\overline{J}^d_{\sigma(C^{\Gamma})}(C^{\Gamma})) \\=  \begin{cases}c(\Gamma)\chi_{orb}(\mathcal{M}^{\Gamma})\;\;\; \text{if }\Gamma \in G(g,n)^0\\
0\;\;\;\;\;\;\;\;\;\;\;\;\;\;\;\;\;\;\;\;\;\; \text{otherwise}.\end{cases}
\end{align*}
So \(\chi_{orb}(\overline{\mathcal{J}}^d_{g,n}) = \underset{\Gamma \in G(g,n)^0}{\sum}c(\Gamma)\chi_{orb}(\mathcal{M}^{\Gamma})\).
\end{proof}

\begin{theorem}\label{jacobian eec formula}
For all \(g\) and \(n\) such that \(2g-2+n>0\),
\begin{align*}
\chi_{orb}(\overline{\mathcal{J}}^d_{g,n}) = \frac{1}{2^g(g!)}\chi(\overline{\mathcal{M}}_{0,2g+n})
\end{align*}
where \(\chi(\overline{\mathcal{M}}_{0,2g+n})\) is the ordinary topological Euler characteristic of the variety \(\overline{\mathcal{M}}_{0,2g+n}\).
\end{theorem}

\begin{proof}
Let \(\widehat{G}(g,n)^0\) be a set of representatives of isomorphism classes of pairs \((\Gamma, T)\) such that \([\Gamma] \in G(g,n)^0\) and \(T\) is a spanning tree of \(\Gamma\). An isomorphism of \((\Gamma,T)\) is an automorphism of \(\Gamma\) fixing \(T\) and we denote this group by \(\mathrm{Aut}(\Gamma,T)\). The group \(\mathrm{Aut}(\Gamma)\) acts on the set of spanning trees of \(\Gamma\). The stabilizer of a fixed spanning tree \(T\) under this action is given by \(\mathrm{Aut}(\Gamma,T)\). Therefore,
\begin{align*}
\chi_{orb}(\overline{\mathcal{J}}^d_{g,n}) = \sum_{\Gamma \in G(g,n)^0}c(\Gamma)\chi_{orb} (\mathcal{M}^{\Gamma})\\=\sum_{\Gamma \in G(g,n)^0}\sum\limits_{\substack{T \text{ is a}\\\text{spanning}\\ \text{tree of }\Gamma}}\chi_{orb}(\mathcal{M}^{\Gamma}) \\ = \sum_{(\Gamma, T) \in \widehat{G}(g,n)^0}|\mathrm{orb}(T)|\chi_{orb}(\mathcal{M}^{\Gamma}).
\end{align*}
Additionally, by \cite{BiHa}:
\begin{align*}
\chi_{orb}(\mathcal{M}^{\Gamma}) = \frac{\prod_{v \in V(\Gamma)}\chi_{orb}(\mathcal{M}_{0,n(v)})}{|\mathrm{Aut}(\Gamma)|}
\end{align*}
We combine this with the fact that by the orbit stabilizer theorem, \(|\mathrm{Aut}(\Gamma)| = |\mathrm{Aut}(\Gamma, T)||\mathrm{orb}(T)|\) (where \(T\) is any spanning tree of \(\Gamma\)) to obtain:
\begin{align*}
\chi_{orb}(\overline{\mathcal{J}}^d_{g,n}) = \sum_{(\Gamma, T) \in \widehat{G}(g,n)^0}|\mathrm{orb}(T)|\frac{\prod_{v \in V(\Gamma)}\chi_{orb}(\mathcal{M}_{0,n(v)})}{|\mathrm{Aut}(\Gamma)|} \\= \sum_{(\Gamma, T) \in \widehat{G}(g,n)^0}\frac{\prod_{v \in V(\Gamma)}\chi_{orb}(\mathcal{M}_{0,n(v)})}{|\mathrm{Aut}(\Gamma, T)|}.
\end{align*}
\noindent
We define a surjective function \(\phi: G(0, 2g+n) \rightarrow \widehat{G}(g,n)^0\). Gluing pairs of markings \(n+2i-1, n+2i\) for \(1 \leq i \leq g\) of \(T' \in G(g,n)^0\) forms a stable graph \(\Gamma\) with \(n\) legs of genus \(g\) and the images of the original edges in \(T'\) result in a spanning tree \(T\) or \(\Gamma\). We define \(\phi(T') = (\Gamma, T)\). We have \(|\phi^{-1}([(\Gamma,T)])| = \frac{2^g(g!)}{|\mathrm{Aut}(\Gamma,T)|}\) which can be seen as follows.\\
Fixing \((\Gamma,T)\) in \(\hat{G}(g,n)^0\), any \(T'\) in the preimage can be obtained by matching pairs of half-edges which form an edge in \(E(\Gamma)\setminus E(T)\) to the entries of the \(g\) pairs \((n+1,n+2),...,(2g+n-1,2g+n)\). Then \(T'\) is obtained by cutting these edges and identifying the corresponding half-edges with the markings following the assignments previously described. There are \(g!\) ways to match pairs of half-edges in \(E(\Gamma)\setminus E(T)\) with pairs \((n+1,n+2),...,(2g+n-1, 2g+n)\). Once the pairs of half-edges have been assigned, there are \(2^g\) ways of determining the ordering of the half-edges within the pairs. Let \(S\) be the set of these \(2^g(g!)\) choices. Then there is a surjective map \(S \rightarrow \phi^{-1}(\Gamma,T)\) where the image can be identified with the orbit classes of an action of \(\mathrm{Aut}(\Gamma,T)\) on \(S\). This action is described as follows. If \(T' \in S\) and there is an element \(\sigma\) of \(\mathrm{Aut}(\Gamma,T)\) permuting the half-edges in the edge \((1,2)\), then \(\sigma \cdot T'\) is the element of \(S\) obtained by swapping markings 1 and 2 of \(T'\). This action is free so all the orbits, which partition \(S\) have size \(|\mathrm{Aut}(\Gamma, T)|\) and therefore:
\begin{align*}
|\phi^{-1}((\Gamma,T))| =\frac{|S|}{|\mathrm{Aut}(\Gamma,T)|} =  \frac{2^g(g!)}{|\mathrm{Aut}(\Gamma,T)|}.
\end{align*}
Observing that \(G(0,2g+n)^0 = G(0,2g+n)\), we obtain:
\begin{align*}
\chi_{orb}(\overline{\mathcal{J}}^d_{g,n}) = \sum_{(\Gamma, T) \in \widehat{G}(g,n)^0}\sum_{\lbrace T' \in \phi^{-1}((\Gamma,T))\rbrace}\frac{\prod_{v \in V(\Gamma)}\chi_{orb}(\mathcal{M}_{0,n(v)})}{|Aut(\Gamma,T)||\phi^{-1}(\Gamma,T)|} \\=\sum_{T' \in G(0, 2g+n)}\frac{\prod_{v \in V(T')}\chi_{orb}(\mathcal{M}_{0,n(v)})}{2^g(g!)}
= \frac{\chi_{orb}(\overline{\mathcal{M}}_{0, 2g + n})}{2^g(g!)}.
\end{align*}
Where in the last two equalities, we use the fact that \(\overline{\mathcal{M}}_{0,2g+n}\) has a locally closed stratification \(\overline{\mathcal{M}}_{0,2g+n} = \underset{T' \in G(0,2g+n)}{\bigsqcup}\mathcal{M}^{T'}\) and the fact that for \(T' \in G(0,2g+n)\), the automorphism group of \(T'\) is trivial and therefore, the gluing morphism \(\xi_{T'}: \prod_{v \in V(T')}\mathcal{M}_{0,n(v)} \rightarrow \mathcal{M}^{T'}\) is an isomorphism onto its image.

\noindent
We conclude using the fact that \(\overline{\mathcal{M}}_{0,2g+n}\) is a projective variety, hence its orbifold Euler characteristic is equal to its topological Euler characteristic.
\end{proof}
\bibliographystyle{alpha}
\bibliography{references}
\end{document}